\documentclass[12pt]{article}
\usepackage{amsmath, amsthm}
\usepackage{amscd}
\usepackage{amssymb}
\usepackage{array}
\usepackage{color}
\usepackage[all,cmtip]{xy}

\newtheorem{thm}{Theorem}[section]
\newtheorem{theorem}[thm]{Theorem}

\newtheorem{corollary}[thm]{Corollary}

\newtheorem{proposition}[thm]{Proposition}

\theoremstyle{definition}
\newtheorem{definition}[thm]{Definition}

\newtheorem{remark}[thm]{Remark}

\begin{document}

%\newcommand{\comment}[1]
%{{\color{blue}\rule[-0.5ex]{2pt}{2.5ex}}
%\marginpar{\small\begin{flushleft}\color{blue}#1\end{flushleft}}}
%\newcommand{\comment}[1]
%{{\color{blue}\rule[-0.5ex]{2pt}{2.5ex}}
%\marginpar{\small\begin{flushleft}\color{blue}#1\end{flushleft}}}
%\newcommand{\alert}[1]{{\color{red}#1}}

%\newcommand{\green}[1]{{\textbf{#1}}}
%\newcommand{\Red}[1]{{\color{red}{#1}}}

\newcommand{\id}{\relax{\rm 1\kern-.28em 1}}
\newcommand{\R}{\mathbb{R}}
\newcommand{\C}{\mathbb{C}}
\newcommand{\Cbar}{\overline{\C}}
\newcommand{\Z}{\mathbb{Z}}
\newcommand{\Q}{\mathbb{Q}}
\newcommand{\kk}{k}%\mathbb{k}}
\newcommand{\bD}{\mathbb{D}}
\newcommand{\bG}{\mathbb{G}}
\newcommand{\bP}{\mathbb{P}}
\newcommand{\bM}{\mathbb{M}}
\newcommand{\g}{\mathfrak{G}}
\newcommand{\fh}{\mathfrak{h}}
\newcommand{\e}{\epsilon}
\newcommand{\Uu}{\mathfrak{U}}

\newcommand{\cA}{\mathcal{A}}
\newcommand{\cB}{\mathcal{B}}
\newcommand{\cC}{\mathcal{C}}
\newcommand{\cD}{\mathcal{D}}
\newcommand{\cI}{\mathcal{I}}
\newcommand{\cL}{\mathcal{L}}
\newcommand{\cK}{\mathcal{K}}
\newcommand{\cO}{\mathcal{O}}
\newcommand{\cG}{\mathcal{G}}
\newcommand{\cJ}{\mathcal{J}}
\newcommand{\cF}{\mathcal{F}}
\newcommand{\cP}{\mathcal{P}}
\newcommand{\cU}{\mathcal{U}}
\newcommand{\ep}{\mathcal{E}}
\newcommand{\E}{\mathcal{E}}
\newcommand{\cH}{\mathcal{O}}
\newcommand{\cV}{\mathcal{V}}
\newcommand{\cPO}{\mathcal{PO}}
\newcommand{\cHol}{\mathrm{H}}

\newcommand{\rGL}{\mathrm{GL}}
\newcommand{\rSU}{\mathrm{SU}}
\newcommand{\rSL}{\mathrm{SL}}
\newcommand{\rPSL}{\mathrm{PSL}}
\newcommand{\rSO}{\mathrm{SO}}
\newcommand{\rOSp}{\mathrm{OSp}}
\newcommand{\rSpin}{\mathrm{Spin}}
\newcommand{\rsl}{\mathrm{sl}}
\newcommand{\rsu}{\mathrm{su}}
\newcommand{\rM}{\mathrm{M}}
\newcommand{\rdiag}{\mathrm{diag}}
\newcommand{\rP}{\mathrm{P}}
\newcommand{\rdeg}{\mathrm{deg}}
\newcommand{\pt}{\mathrm{pt}}
\newcommand{\red}{\mathrm{red}}

\newcommand{\bm}{\mathbf{m}}

\newcommand{\M}{\mathrm{M}}
\newcommand{\End}{\mathrm{End}}
\newcommand{\Hom}{\mathrm{Hom}}
\newcommand{\diag}{\mathrm{diag}}
\newcommand{\rspan}{\mathrm{span}}
\newcommand{\rank}{\mathrm{rank}}
\newcommand{\Gr}{\mathrm{Gr}}
\newcommand{\ber}{\mathrm{Ber}}

\newcommand{\str}{\mathrm{str}}
\newcommand{\Sym}{\mathrm{Sym}}
\newcommand{\tr}{\mathrm{tr}}
\newcommand{\defi}{\mathrm{def}}
\newcommand{\Ber}{\mathrm{Ber}}
\newcommand{\spec}{\mathrm{Spec}}
\newcommand{\sschemes}{\mathrm{(sschemes)}}
\newcommand{\sschemeaff}{\mathrm{ {( {sschemes}_{\mathrm{aff}} )} }}
\newcommand{\rings}{\mathrm{(rings)}}
\newcommand{\Top}{\mathrm{Top}}
\newcommand{\sarf}{ \mathrm{ {( {salg}_{rf} )} }}
\newcommand{\arf}{\mathrm{ {( {alg}_{rf} )} }}
\newcommand{\odd}{\mathrm{odd}}
\newcommand{\alg}{\mathrm{(alg)}}
\newcommand{\sa}{\mathrm{(salg)}}
\newcommand{\sets}{\mathrm{(sets)}}
\newcommand{\smflds}{\mathrm{(smflds)}}
\newcommand{\mflds}{\mathrm{(mflds)}}
\newcommand{\shcps}{\mathrm{(shcps)}}
\newcommand{\sgrps}{\mathrm{(sgrps)}}
\newcommand{\SA}{\mathrm{(salg)}}
\newcommand{\salg}{\mathrm{(salg)}}
\newcommand{\varaff}{ \mathrm{ {( {var}_{\mathrm{aff}} )} } }
\newcommand{\svaraff}{\mathrm{ {( {svar}_{\mathrm{aff}} )}  }}
\newcommand{\ad}{\mathrm{ad}}
\newcommand{\Ad}{\mathrm{Ad}}
\newcommand{\pol}{\mathrm{Pol}}
\newcommand{\Lie}{\mathrm{Lie}}
\newcommand{\Proj}{\mathrm{Proj}}
\newcommand{\rGr}{\mathrm{Gr}}
\newcommand{\rFl}{\mathrm{Fl}}
\newcommand{\rPol}{\mathrm{Pol}}
\newcommand{\rdef}{\mathrm{def}}
\newcommand{\ah}{\mathrm{ah}}
\newcommand{\autsusy}{\mathrm{Aut_{SUSY}}}
\newcommand{\aut}{\mathrm{Aut}}
\newcommand{\rPGL}{\mathrm{PGL}}

\newcommand{\uspec}{\underline{\mathrm{Spec}}}
\newcommand{\uproj}{\mathrm{\underline{Proj}}}

\newcommand{\sym}{\cong}

\newcommand{\al}{\alpha}
\newcommand{\be}{\beta}
\newcommand{\lam}{\lambda}
\newcommand{\de}{\delta}
\newcommand{\ttau}{\tilde \tau}
\newcommand{\D}{\Delta}
\newcommand{\s}{\sigma}
\newcommand{\lra}{\longrightarrow}
\newcommand{\ga}{\gamma}
\newcommand{\ra}{\rightarrow}

\newcommand{\wbar}{\overline{w}}
\newcommand{\zbar}{\overline{z}}
\newcommand{\fbar}{\overline{f}}
\newcommand{\etabar}{\overline{\eta}}
\newcommand{\zetabar}{\overline{\zeta}}
\newcommand{\betabar}{\overline{\beta}}
\newcommand{\albar}{\overline{\alpha}}
\newcommand{\abar}{\overline{a}}
\newcommand{\dbar}{\overline{d}}
\newcommand{\tbar}{\overline{t}}
\newcommand{\tX}{\widetilde{X}}
\newcommand{\thetabar}{\overline{\theta}}
\newcommand{\Mbar}{{\overline{M}}}
\newcommand{\aubar}{\underline{a}}
\newcommand{\gabar}{\overline{\gamma}}
\newcommand{\fg}{\mathfrak{g}}
\newcommand{\Span}{\mathrm{span}}
\newcommand{\rSpO}{\mathrm{SpO}}
\newcommand{\Xhat}{\widehat{X}}

\newcommand{\NOTE}{\bigskip\hrule\medskip}

\newcommand{\G}{{(\C^{1|1})}^\times}
\newcommand{\pair}[2]{\langle \, #1, #2\, \rangle}
\newcommand{\cinfty}{\mathcal{C}^\infty}

\newcommand{\beq}{\begin{equation}}
\newcommand{\eeq}{\end{equation}}

\medskip

\centerline{\Large \bf   SUSY $N$-supergroups and their real forms}

\medskip

\bigskip

\centerline{R. Fioresi$^\flat$, S. Kwok$^\star$}

\medskip

\centerline{\it $^\flat$ Dipartimento di Matematica, Universit\`{a} di
Bologna }
 \centerline{\it Piazza di Porta S. Donato, 5. 40126 Bologna. Italy.}
\centerline{{\footnotesize e-mail: rita.fioresi@UniBo.it}}

\medskip

\centerline{\it $^\star$ Mathematics Research Unit, University of Luxembourg}
 \centerline{\it 6, Rue Richard Coudenhove-Kalergi, L-1359, Luxembourg.}
\centerline{{\footnotesize e-mail: 
stephen.kwok@uni.lu, sdkwok2@gmail.com}}

\bigskip

\begin{abstract}
We study SUSY $N$-supergroups, $N=1,2$,
their classification and explicit realization, together
with their real forms. In the end, we give the 
supergroup of SUSY preserving automorphism
of $\C^{1|1}$ and we identify it with a subsupergroup of the
SUSY preserving automorphisms of $\bP^{1|1}$.
\end{abstract}

\section{Introduction} 
\label{intro-sec}
The papers \cite{cfk1} and \cite{cfk2} carry out a thorough study
of the real compact supergroups $S^{1|1}$ and $S^{1|2}$,  
called \textit{supercircles}, in odd dimension $1$ and $2$, and their theory
of representation, together with the Peter-Weyl theorem. These supercircles
are realized as real forms of $(\C^{1|1})^\times$ and $(\C^{1|2})^\times$
respectively, and in the case of $S^{1|1}$, we have a precise
relation between the real structures and real forms of $(\C^{1|1})^\times$
and the SUSY preserving automorphism of the SUSY 1-curve $(\C^{1|1})^\times$.
In this paper we want to study the SUSY
$N$-curves, which also admit a supergroup structure leaving
invariant their SUSY structure, namely
the SUSY $N$-supergroups. 
SUSY $N$-curves have been the object of study of several papers.
After the foundational work by Manin \cite{ma2}, in \cite{rabin1, rabin2} 
Rabin et al. study families of super elliptic curves over
non trivial odd bases, which are
SUSY $1$-curves, whose reduced part is an elliptic curve. These 
families of supercurves however do not admit a natural
supergroup structure leaving invariant
their SUSY structure, hence they are not SUSY $N$-supergroups according to
our terminology (see also Remark \ref{rabin}). 
On the other hand, studying the
real forms and the supergroup structure of SUSY curves gives the
opportunity of exploiting the representation theory for physical applications.
We also point out that we are not merely studying SUSY-$N$ curves that are also
supergroups, but our requirement that the supergroup structure preserves
the SUSY $N$-structure is quite restrictive and reduces drastically the
possibilities for such supercurves, yet provides a useful local model for 
generic ones.

\medskip
Our paper is organized as follows. 

\medskip
In Sections \ref{susyN-sec}, \ref{sl11-sec} we give 
the definition of SUSY $N$-supergroups 
and we classify them. We also give an interpretation of the supergroup
$\rSL(1|1)$ as the SUSY $2$-curve incidence supermanifold of the
SUSY $1$-curve $(\C^{1|1})^\times$ and its dual. 
In Section \ref{rf-sec} we study
of real forms of SUSY $N$-supergroups of type 1 and classify them.
Finally in Section \ref{c11-sec} we compute the supergroup of 
SUSY preserving automorphism of $\C^{1|1}$. 
%The last two sections completes the study initiated in
%\cite{fk1}.

\medskip
{\bf Acknowledgments}. We wish to thank prof. P. Deligne for
helpful comments. We are indebted to our anonymous Referee,
for valuable comments that helped us to considerably improve our paper.
R. Fioresi, S. D. Kwok thank the University of Luxembourg 
and the University of Bologna for
the wonderful hospitality while this work was done. 
S. D. Kwok was partly supported by AFR grant no. 7718798 of the 
Luxembourgish National Research Foundation.
% while carrying out of much of this work.}

\section{The SUSY $N$-supergroups} 
\label{susyN-sec}

We want to study SUSY $N$-curves which also have a supergroup
structure, which preserves the SUSY $N$-structure.
In the following we shall use the notation and terminology as
in \cite{ma1} Ch. 2. $X$ is a \textit{SUSY $1$-curve} if $X$ is a $1|1$ 
complex supermanifold and there is
a $0|1$ distribution $\cD$ such that the Frobenius map $\cD \otimes \cD \to TX/\cD$ given by $Y_1 \otimes Y_2 \mapsto [Y_1, Y_2] \, (mod \, \cD)$ is an isomorphism.
$X$ is a \textit{SUSY $2$-curve} if it is a $1|2$ 
complex supermanifold and there are two 
$0|1$ distributions $\cD_i$ such that $[\cD_i,\cD_i] \subset \cD_i$
and the Frobenius map $\cD_1 \otimes \cD_2 \to TX/[\cD_1,\cD_2]$ is an isomorphism. In \cite{ma1} Ch. 2,
Manin provides local models for such distributions:
$$
\begin{array}{rlrll}
\cD&=\zeta\partial_z+\partial_\zeta & & & \hbox{ on } X \hbox{ a SUSY 1-curve} \\ \\
\cD_1&=\zeta_2\partial_z+\partial_{\zeta_1} & 
\cD_2&=\zeta_1\partial_z+\partial_{\zeta_2} & \hbox{ on } X \hbox{ a SUSY 2-curve}
\end{array}
$$

\begin{definition}
Let $X$ be a SUSY $N$-curve, with $0|1$ distribution(s) $\cD_i$,
where $i=1$ for $N=1$ and $i=1,2$ for $N=2$.
We say that $X$ is a \textit{SUSY $N$-supergroup} if 
$X$ is a supergroup and the distribution(s) $\cD_i$ are left invariant.

\medskip
If $X$ and $Y$ are SUSY $N$-supergroups, we say that 
$f:X \lra Y$ is a \textit{morphism} of SUSY $N$-supergroups if 
$f$ is a supergroup morphism and
$f_*(\cD_i)=\cD_j$, that is $f$ preserves the distributions or
exchanges them.
\end{definition}

We can immediately compute the Lie superalgebra of SUSY $N$-supergroups.

\begin{proposition} \label{lie-alg}
Let $X$ be a SUSY $N$-supergroup. Then:
$$
\begin{array}{c}
\Lie(X)=\langle Z,C \rangle, \quad [Z,Z]=0, \qquad \hbox{for} \quad N=1 \\ \\
\Lie(X)=\langle Z_1,Z_2,C \rangle, \quad [Z_1,Z_2]=C, \qquad \hbox{for} \quad
N=2
\end{array}
$$
where $Z$, $Z_i$ are suitably chosen odd elements and 
we assume to be zero all the brackets we do not write.
\end{proposition}

\begin{proof}
For $N=1$ see \cite{cfk1}.
Let $N=2$, $\cD_i$ the left invariant distributions on 
the SUSY $2$-supergroup $X$. Let $Z_i
\in \Lie(X)$ be a left invariant (odd) generator of $\cD_i$. We have
$\Lie(X)=\langle Z_1, Z_2, [Z_1,Z_2]\rangle$. Notice that in general
$[\cD_i,\cD_i]\neq 0$, however since $Z_i \in  \Lie(X)$ and the bracket
must preserve the parity, we have $[Z_i,Z_i]=0$. Let $C:=[Z_1,Z_2]$.
The Jacobi identity gives immediately that $C$ is central.
\end{proof}

Since SUSY $N$-supergroups are in particular supergroups,
we shall use freely the formalism of Super Harish Chandra Pairs (SHCP), see
\cite{ccf} Ch. 7 for more details.

\begin{proposition}
Let $X$, $X'$ be SUSY $N$-supergroups. 
$\tX$, $\tX'$ their underlying reduced groups. Then
$X \cong X'$ if and only if $\tX \cong \tX'$.
\end{proposition}

\begin{proof}
Suppose $f: X \to X'$ is an isomorphism.
Then $\tX$ and $\tX'$ are isomorphic. 

Conversely, suppose $|f|: \tX \to \tX'$ is an isomorphism. 
$|f|: \tX \to \tX'$ lifts to the universal covers, 
giving an isomorphism $\widetilde{f}: \C \to \C$ which fixes the origin.
By a standard result from one-variable complex analysis, 
$\widetilde{f}$ is multiplication by a nonzero scalar $\lambda$. 
If $N=1$, define the super Lie algebra morphism 
$\varphi: \fg \to \fg$ by $C \mapsto \lambda C'$, 
$Z \mapsto \sqrt{\lambda} Z'$. Clearly $\varphi|_{\fg_0} = d|f|$ so
 $F:= (|f|, \varphi)$ is an isomorphism of SHCP, hence $X \cong X'$
as supergroups.  
By construction $dF(\cD_e) = \cD'_e$ at the identity $e$. 
However, as $\cD$, $\cD'$ are left-invariant and
$F$ is a supergroup isomorphism, we
have $dF(\cD_p) = \cD'_p$ at every point $p$ of $X$, 
hence $F$ is an isomorphism of SUSY $N$-supergroups.

If $N=2$, define the super Lie algebra morphism 
$\varphi: \fg \to \fg$ by $C \mapsto \lambda C'$, 
$Z_1 \mapsto \sqrt{\lambda} Z_1'$, $Z_2 \mapsto \sqrt{\lambda} Z_2'$. 
Again, $\varphi|_{\fg_0} = d|f|$ so $F := (|f|, \varphi)$ 
is an isomorphism of SHCP. This shows that $X \cong X'$ as
supergroups, however, reasoning as above, we obtain our result.            
\end{proof}

%\begin{observation} \label{sgrp-iso}
%If $X$ and $Y$ are SUSY $N$-supergroups, which are isomorphic
%as supergroups, then $X$, $Y$ are also isomorphic as SUSY $N$-supergroups.
%This happens because, since $X$ and $Y$ are isomorphic as supergroups
%they have isomorphic Lie superalgebras. Hence we can construct
%the left invariant distributions $\cD_i$ via the left translation in
%both supergroups.
%\end{observation}

\begin{corollary}\label{cov-prop}
Let $X$ be a SUSY $N$-supergroup. Then, $X=\C^{1|N}$ 
or
$X \cong \C^{1|N}/G$, where $G$ is either  
a rank $1$ free abelian subgroup of $\C^{1|N}$, 
or a lattice in $\C^{1|N}$.
Furthermore, $X$ inherits its SUSY $N$-structure from
its universal cover $\C^{1|N}$.
\end{corollary}

\begin{proof}
If $X$ is simply connected, then $X=\C^{1|N}$. Assume $X$ is not simply
connected. 
The universal cover of $X$ is readily
seen to be $\C^{1|N}$. The kernel of the covering morphism 
$\pi: \C^{1|N} \to X$ is a $0|0$-subsupergroup $G$ of $\C^{1|N}$ and,
by a classical argument, is 
either a rank $1$ free abelian subgroup of $\C^{1|N}$, 
or a lattice in $\C^{1|N}$. So $X \cong \C^{1|N}/G$. 

%As for the claim on the $N$-supergroup structure: the local
%expression for the SUSY $N$-structure on $X$ lifts to give a 
%global SUSY $N$-structure on $\C^{1|N}$.
%In the first case, that is for $G$ free abelian
%of rank $1$, a straightforward calculation shows that 
%the translation by $G$ leaves the SUSY $N$ vector field(s) 
%$\partial_\zeta + \zeta \partial_z$ 
%(resp. $\partial_{\zeta_1} + \zeta_2 \partial_z$, 
%$\partial_{\zeta_2} + \zeta_1 \partial_z$) invariant, so the SUSY 
%$N$-structure on $\C^{1|N}$ descends to $X= \C^{1|N}/G$. 
%$X$ is then easily seen to be a SUSY $N$-supergroup with underlying 
%reduced group $\C^\times$, i.e. it is isomorphic to $(\C^{1|N})^\times$.

%In the second case, 
%we immediately see, as before, that
%the translation by $G$ preserves the SUSY $N$-structure on $\C^{1|N}$, 
%hence such structure descends to $X$; 
%then $X= \C^{1|N}/G$ is a SUSY $N$-supergroup whose underlying 
%reduced group is an elliptic curve.
\end{proof}

\begin{definition}
We say that $X$ is a SUSY $N$-supergroup of \textit{type 1} (resp. 
\textit{type 2})
if $X=\C^{1|N}/G$, where $G$ is a rank $1$ free abelian subgroup 
(resp. a lattice) in $\C^{1|N}$. 
\end{definition}

%We focus now on the type 2 SUSY $N$-supergroups, leaving the
%discussion of type 1 for the next section. 

We end this section with a remark relating our treatment with 
\cite{rabin1, rabin2}.

\begin{remark}\label{rabin}
In \cite{rabin1, rabin2},  Rabin et al. 
consider families $X \to \cB$ of SUSY $1$-elliptic 
curves over a base superspace $\cB = (pt, \Lambda)$, where $\Lambda$ 
is a non trivial Grassmann algebra. In this setting, 
families of SUSY 1-curves do not admit any natural 
structure of supergroup {\it over $\cB$}. 
%(``a superelliptic curve does not carry the 
%group structure of its Jacobian", p. 255 {\it loc. cit.}) 

This is consistent with our treatment, because
we work in the {\sl absolute} setting, i.e. taking $\cB = \{pt\}$, 
so it is possible 
to endow a SUSY $1$-elliptic curve with the supergroup structure 
inherited from its universal cover $\C^{1|1}$ (Prop. \ref{cov-prop}).
\end{remark}

\section{SUSY $N$-supergroups of type 1}\label{sl11-sec}

We want to classify the SUSY $N$-supergroups of type 1 and relate them
to Manin's approach to SUSY curves. We shall use interchangeably
the formalisms of functor of points and also of
super Harish-Chandra pairs (SHCP).
%, namely we describe
%a supergroup by a pair $(G_0, \fg)$, where $G_0$ is a complex group
%and $\fg$ is a Lie superalgebra (see \cite{ccf} Ch. 3, 7). 

\begin{proposition} \label{class-N-sgrps}
Up to isomorphism, for $N=1,2$ fixed, we have only two SUSY $N$-supergroups
of type 1.
\begin{enumerate}
\item For $N=1$ they are $(\C^{1|1})^\times$ and $\C^{1|1}$ with group law
respectively:
\beq\label{grp-law1}
\begin{array}{rl}
(w,\eta)\cdot(w,\eta')&=(ww'+\eta\eta', w\eta'+\eta w') \\ \\
(z,\zeta)\cdot(z',\zeta')&=(z+z'+\zeta\zeta', \zeta+\zeta')
\end{array}
\eeq

\item For $N=2$ they are $(\C^{1|2})^\times$ and $\C^{1|2}$ with group laws: 
%$$

\beq\label{grp-law2}
\begin{array}{rl}(v,\xi, \eta) \cdot (v',\xi', \eta')&=(vv'+\eta\xi',
v\xi'+\xi v' + \xi v^{-1}\eta\xi', \\ \\ 
&\eta v' + v \eta' + \eta \xi' v'^{-1} \eta') \\ \\
(z,\zeta,\chi)\cdot(z',\zeta',\chi')&=(z+z'+\zeta\chi', \zeta+\zeta',\chi+\chi')
\end{array}
%$$
\eeq
\end{enumerate}
\end{proposition}

\begin{proof} For $N=1$ the statements are contained in \cite{cfk1},
provided that one verifies left invariance, which is a straightforward
check. Let $N=2$, $\cD_i$ the left invariant distributions on 
the SUSY $2$-supergroup $X$. Let $D_i
\in \Lie(X)$ be a left invariant (odd) generator of $\cD_i$. 
By \ref{lie-alg} we have
$\Lie(X)=\langle D_1, D_2, [D_1,D_2]\rangle$. 
%Notice that in general
%$[\cD_i,\cD_i]\neq 0$, however since $D_i \in  \Lie(X)$ and the bracket
%must preserve the parity, we have $[D_i,D_i]=0$. Let $C:=[D_1,D_2]$.
%The Jacobi identity gives immediately that $C$ is central and then
%an easy calculation shows that 
The given group laws correspond to the
Lie superalgebra we have obtained respectively for $G_0=\C^\times$ and
$G_0=\C$. For example let us compute the Lie superalgebra structure
for $(\C^{1|2})^\times$ (the case of $\C^{1|2}$ is similar). The tangent
space is $\C^{1|2}$ and let $e_1$, $\ep_1$, $\ep_2$ be the
canonical basis.
The corresponding left invariant vector fields in the 
global coordinates $(v,\xi,\eta)$ are:
$$
\begin{array}{rl}
D_1&=(d \ell_{(u,\mu,\nu)})_{(1,0,0)}\ep_1=v \partial_\eta \\ \\
D_2& =(d \ell_{(u,\mu,\nu)})_{(1,0,0)}\ep_2=-\eta \partial_v+
(v+\xi v^{-1}\eta)\partial_\xi \\ \\
E&=(d \ell_{(u,\mu,\nu)})_{(1,0,0)}e_1=v \partial_v+
\xi \partial_\xi+\eta\partial_\eta 
\end{array}
$$
As one can readily check $[D_1,D_2]=-2E$ (so we set $C=-2E$) and
$[D_i,D_i]=0$. 
%So we have obtained our result.
\end{proof} 

\begin{remark}\label{interpret}
We can interpret the multiplicative and additive SUSY $2$-supergroups
using matrix supergroups. $(\C^{1|2})^\times$ is $\rSL(1|1)$. In functor
of points notation:
$$
(\C^{1|2})^\times(T)=\rSL(1|1)(T)=\left\{
\begin{pmatrix} u & \xi \\ \eta & v \end{pmatrix}\, | \, 
v^{-1}(u-\xi v^{-1}\eta)=1\right\}
$$
$\C^{1|2}$ is the subgroup of $\rSL(2|1)$ given in the functor of points
notation by:
$$
\C^{1|2}(T)=\left\{
\begin{pmatrix} 1 & z & \zeta \\ 0 & 1 & 0 \\ 0 & \chi & 1 
\end{pmatrix} \right\}
$$
For example, let us check the second of these statements, the first
one being the same. We can verify the claim reasoning in terms of
the functor of points. So we have to compute:
$$
\begin{pmatrix} 1 & z & \zeta \\ 0 & 1 & 0 \\ 0 & \chi & 1 
\end{pmatrix} \cdot 
\begin{pmatrix} 1 & z' & \zeta' \\ 0 & 1 & 0 \\ 0 & \chi' & 1 
\end{pmatrix}=
\begin{pmatrix} 1 & z'+z+\zeta\chi' & \zeta+\zeta' \\ 0 & 1 & 0 \\ 0 & 
\chi+\chi' & 1 
\end{pmatrix}
$$
which is precisely the multiplication as in (\ref{grp-law2}).
This approach can be helpful in calculations. We leave to the reader
the straightforward checks regarding the group law of the first statement.
\end{remark}

We now want to interpret some of the discussion in \cite{ma1} Sec. 6
in the framework of SUSY supergroups.
Let $X$ be a SUSY $1$-curve
and $\widehat{X}$ its dual.
The $T$-points of $\Xhat$ are the $0|1$ subvarieties of $X(T)$.
Let $\Delta$ be the superdiagonal subscheme of $X \times  \Xhat$. It
is locally defined by
the incidence relation:
\beq\label{inc-rel}
z-z'-\zeta'\zeta=0
\eeq
$(z,\zeta')$ and $(z',\zeta')$ local coordinates of $X$ and $\Xhat$
(see Def. 6.2 in \cite{ma1}). 
$\Delta$ is a SUSY 2-curve, with distributions $\cD_1$, $\cD_2$
and we have the commutative diagram: 
\beq \label{diagr1}
\xymatrix{& \Delta \ar[rd] \ar[ld]  &\\
X=\Delta/\cD_1 &  \ar[r]   & \Delta/\cD_2=\Xhat
%\begin{array}{ccccc}
%& &\Delta \subset X \times \Xhat  &\\ \\
%&\swarrow & & \searrow & \\ \\
%X &=\Delta/\cD_1  &\cong  & \Delta/\cD_2=\Xhat
%\end{array}
}
\eeq
where $\Delta/\cD_i$ means the superspace whose reduced space is $|\Delta|$, and whose structure sheaf is the subsheaf of $\cO_\Delta$
consisting of sections which are invariant under $\cD_i$. Specifying a SUSY-1 structure on $X$ gives an isomorphism $X \cong \Xhat$. 
%Furthermore, we have an involution $s:\Delta \lra \Delta$, such that
%$s_*(\cD_1)=\cD_2$ and this involution 
%induces the isomorphism $X \cong \Xhat$. 

\medskip
On $(\C^{1|2})^\times$ we have the global SUSY 2-structure:
\begin{equation}\label{2susy}
D_1=\partial_{\zeta_1}+\zeta_2 \partial_z \quad
D_2=\partial_{\zeta_2}+\zeta_1 \partial_z
\end{equation}
Clearly $D_1^2=D_2^2=0$ and $[D_1,D_2]=2\partial_z$. In Remark
\ref{interpret}, we have viewed $(\C^{1|2})^\times$ as the supergroup
$\rSL(1|1)$. The
condition on the berezinian 
is the (global) incidence relation (\ref{inc-rel})
and allow us to identify $\rSL(1|1)$ with $\Delta$
for $X=(\C^{1|1})^\times$ and $\rGL(1|1)$ with $X \times \Xhat$.
%where $X\cong\Xhat$ is describe by item 1. in Prop. \ref{class-N-sgrps}.
Notice that 
in our special case, namely for
$X=\Xhat=(\C^{1|1})^\times$, we have
that  $X,\Xhat \subset \rSL(1|1)$ as the subsupergroups:
$$
%\begin{array}{rl}
X(T)=\left\{
\begin{pmatrix} x & \xi \\ \xi & x \end{pmatrix}\,\right\}, \qquad%\\ \\
\Xhat(T)=\left\{
\begin{pmatrix} y & \eta \\ -\eta & y \end{pmatrix}\,  \right\}
%\end{array}
$$
These inclusions correspond to the Lie superalgebra inclusions:
$$
\langle C, U=E+F \rangle , \quad \langle C, V=E-F \rangle \, \subset \,
\langle C,E,F \rangle=\rsl(1|1)
$$
where $C$, $E$, $F$ are the usual generators for  $\rsl(1|1)$, namely:
\beq\label{pres-sl}
[C,E]=[C,F]=[E,E]=[F,F]=0, \qquad [E,F]=C
\eeq
%It is important to remark that while the $X$ and $\Xhat$ embed into
%$\Delta=\rSL(1|1)$ as its subsupergroups, the arrows in (\ref{diagr1})
%are not supergroup morphisms.

\section{Real forms of SUSY supergroups} \label{rf-sec}

We want to study the real forms of SUSY $N$-supergroups of type 1.
In \cite{cfk1} and \cite{cfk2} we proved that, up to isomorphism,
there is one real form of the SUSY $1$-supergroup $(\C^{1|1})^\times$ 
and the corresponding involution is the composition of 
complex conjugation and the SUSY preserving automorphisms $P_\pm$.
We wish to prove a similar result for the SUSY $2$-supergroups.

\begin{definition}
Let $X$ be a SUSY $2$-curve with distributions $\cD_i$. We say
that an automorphism $\phi:X\lra X$ is \textit{SUSY preserving}
if $\phi_*(\cD_i)=\cD_j$, that is, if $\phi$ preserves individually
each distribution or exchanges them. If $X$ is a SUSY $2$-supergroup
we additionally require  $\phi$ to be a supergroup automorphism.
\end{definition}

Notice that in a SUSY $2$ curve the roles of $\cD_1$ and $\cD_2$
are interchangeable; this forces us to give such a definition of SUSY preserving
automorphism.

\medskip
We start our discussion by observing that, up to isomorphism, there
is only one real form of the Lie superalgebra $\rsl(1|1)$ with compact
even part. In fact, assume $\fg_\R=\Span_\R \{iC,U,V \}$ is such
real form, with central even element $iC$ (see (\ref{pres-sl})). 
If $U=aE+bF$, $V=cE+dF$, there is no loss of generality
in assuming $a=1$ because $E \mapsto a^{-1}E$, $F \mapsto aF$, $C \mapsto C$
is a Lie superalgebra automorphism of $\rsl(1|1)$. 
Assume furtherly $[U,U]\neq 0$
(when both $[U,U]=[V,V]=0$ we leave to the reader the easy check
of what the real form is).
Then we have that $b=i$, up to a constant, that we absorbe in $C$. 
An easy calculation
shows that $V=iE-F$, hence we have proven the following proposition.

\begin{proposition} \label{sl-rf}
Up to isomorphism, there is a unique real form of $\rsl(1|1)$ with compact
even part, namely
$$
\rsu(1|1)=\Span_\R \{ iC,U=E+iF,V=iE-F \}\subset \rsl(1|1).
$$
\end{proposition}

We can then state the theorem giving all real forms of SUSY
$2$-supergroups with compact support.

\begin{theorem} Up to isomorphism,
there exists a unique real form of the SUSY $2$-supergroup $\rSL(1|1)$ and
it is obtained with an involution $\sigma=c \circ \phi$ where $\phi$ is
a SUSY preserving automorphism and $c$ is a complex conjugation.
Explicitly:
$$
\sigma\begin{pmatrix} a & \beta \\ \gamma  & d \end{pmatrix} \, = \,
\begin{pmatrix} \dbar^{-1} & -i\abar^{-2}\gabar \\ 
-i\abar^{-2}\betabar  & \abar^{-1}\end{pmatrix}
$$
\end{theorem}

\begin{proof}
We first check that the given $\sigma$ is of the prescribed type,
namely that 
$$
\phi\begin{pmatrix} a & \beta \\ \gamma  & d \end{pmatrix} \, = \,
\begin{pmatrix} d^{-1} & -ia^{-2}\gamma \\ 
-ia^{-2}\beta  & a^{-1}\end{pmatrix}
$$
is a SUSY preserving automorphism. 
The fact that $\phi$ is a supergroup automorphism
is a simple check, that can be verified using the functor of point
formalism, namely one sees that:
$$
\phi\left( \begin{pmatrix} a & \beta \\ \gamma  & d \end{pmatrix} \cdot
\begin{pmatrix} a' & \beta' \\ \gamma'  & d' \end{pmatrix} \right) \, = \,
\begin{pmatrix} d^{-1} & -ia^{-2}\gamma \\ 
-ia^{-2}\beta  & a^{-1}\end{pmatrix}
\begin{pmatrix} d'^{-1} & -ia'^{-2}\gamma \\ 
-ia'^{-2}\beta  & a'^{-1}\end{pmatrix}
$$
Since $\phi$ is a supergroup morphism, $d\phi$ preserves left invariant
vector fields (see \cite{ccf} Ch. 7), hence it preserves the
SUSY structure.
% can be
%verified at the Lie superalgebra level and we leave it to the reader.}

As for uniqueness, by Prop. \ref{sl-rf} we know there exists a unique
real form $\rsu(1|1)$ of $\rsl(1|1)=\Lie(\rSL(1|1))$ and one readily checks 
$\rsu(1|1)=\Lie(\rSU(1|1))$. By the equivalence of categories in SHCP
theory we obtain the result.
\end{proof}

\section{The SUSY preserving automorphisms of $\C^{1|1}$} 
\label{c11-sec}

Let our notation and terminology be as in \cite{cfk1}, \cite{cfk2}.

\medskip
On $\C^{1|1}$ we have the globally defined SUSY structure 
given by the vector field:
$$
D=\partial_\zeta+\zeta \partial_z
$$
where $(z,\zeta)$ are the global coordinates.
This structure is unique up to isomorphism (see \cite{fk1} Sec. 4).
We want to determine the supergroup of automorphism of
$\C^{1|1}$ preserving such
SUSY structure. We will denote it with 
$\autsusy(\C^{1|1})$.  In the work \cite{fk1} we have
provided the $\C$-points of such supergroup; they are obtained by 
looking at the transformations leaving invariant the $1$-form:
$$
s=dz-\zeta d \zeta
$$
and are given by the endomorphisms
$$
F(z,\zeta)=(az+b, \sqrt{a} \zeta)
$$
We can identify the $\C$-points of the supergroup of SUSY preserving
automorphism with the matrix group
\begin{equation}\label{mgrp}
\autsusy(\C^{1|1})(\C)=\left\{ \begin{pmatrix} c & d & 0\\ 0 & c^{-1} & 0 \\
0&0& 1\end{pmatrix} \, | \, a,b \in \C\right\} \subset
\autsusy(\bP^{1|1})(\C) 
\end{equation}
This is a subgroup of the $\C$-points of the SUSY-preserving automorphisms
of the SUSY 1-curve $\bP^{1|1}$, namely those  fixing the
point at infinity
(see Sec. 5 \cite{fk1} and Sec. 5). In such
identification $a=c^2$ and $b=dc$.
Notice that, though 
$\autsusy(\C^{1|1})(\C)$ is a matrix group,
it is not obvious that also $\autsusy(\C^{1|1})$ should be, since
we are looking at the SUSY preserving automorphism of $\C^{1|1}$
as supermanifold morphisms. Neverthless we will show that this is the case.

\medskip
An automorphism $F: \C^{1|1} \lra \C^{1|1}$ 
induces an automorphism $F^*:\cO(\C^{1|1}) \lra \cO(\C^{1|1})$
of the superalgebra of global sections. $F$ is SUSY preserving
if and only if 
\beq \label{comm}
F^* \circ D=kD \circ F^*
\eeq
where $D$ is now interpreted as a derivation of $\cO(\C^{1|1})$
and $k$ is a suitable constant.
We first consider the infinitesimal picture
and compute \break $\Lie(\autsusy(\C^{1|1}))$. By (\ref{mgrp}),
$\Lie(\autsusy(\C^{1|1}))_0$ is $2$ dimensional, and as one can
readily check, it is spanned by the two even vector fields:
$$
U_1=2z\partial_z+ \zeta \partial_\zeta,
\quad
U_2=\partial_z
$$
We hence only need to compute $\Lie(\autsusy(\C^{1|1}))_1$.

\begin{proposition} \label{inf-aut}
$\Lie(\autsusy(\C^{1|1}))$ is
the Lie subsuperalgebra of the vector fields on $\C^{1|1}$
spanned by
$$
U_1=2z\partial_z+ \zeta \partial_\zeta,
\quad
U_2=\partial_z, \quad
V=\zeta \partial_z - \partial_\zeta.
$$
with brackets:
$$
[V,V]=2U_2, \quad [U_2,U_1]=-2U_1, \quad [U_2, V]=-V,\quad [U_1, V]=0
$$
\end{proposition}

\begin{proof} 
Consider $I+\theta\chi$, for $\chi \in \Lie(\autsusy(\C^{1|1}))_1$.
The condition (\ref{comm}) gives immediately that the odd derivation
$\chi^*$ of $\cO(\C^{1|1})$ induced by $\chi$ must satisfy
$[\chi^*,D]=0$.
A small calculation gives then the result.

\end{proof}

In the Super Harish-Chandra pair (SHCP) formalism, we can
immediately write the supergroup of SUSY preserving automorphism:
$$
\autsusy(\C^{1|1})=(\autsusy(\C^{1|1})(\C),
\Lie(\autsusy(\C^{1|1}))
$$
The next proposition identifies such supergroup with a natural
subsupergroup of $\autsusy(\bP^{1|1})=\rSpO(2|1)$ (Ref. \cite{fk2})
using the more geometric functor of points approach.

\begin{proposition}
$\autsusy(\C^{1|1})$ is the stabilizer subsupergroup in $\rSpO(2|1)$
of the point at infinity:
$$
\autsusy(\C^{1|1})(T)=\left\{ \begin{pmatrix} c & d & \gamma\\ 0 & c^{-1} & 0 \\
0&c^{-1}\gamma & 1\end{pmatrix} \, | \, c,d \in \cO(T)_0, \, \gamma \in 
\cO(T)_1 \right\} %\subset \rSpO(2|1)(T), 
$$
$T \in \smflds_\R$.
\end{proposition}

\begin{proof} The first statement is an immediate consequence
of Proposition \ref{inf-aut}. As for the second one, 
consider the subgroup $G$ of 
$\rSpO(2|1)=\autsusy(\bP^{1|1})$ that fixes the point at infinity.
Its functor of points is given by:
$$
G(T)=\left\{ \begin{pmatrix} c & d & \gamma\\ 0 & c^{-1} & 0 \\
0&c^{-1}\gamma & 1\end{pmatrix} \, | \, c,d \in \cO(T)_0, \, \gamma \in 
\cO(T)_1 \right\} 
$$
$G$ is representable and its SHCP coincides with $\autsusy(\C^{1|1})$,
because $\Lie(G)$ $=$ $\Lie(\autsusy(\C^{1|1})$, $|G|=\autsusy(\C^{1|1})(\C)$
and we have the compatibility conditions. 
\end{proof}


\begin{thebibliography}{99} 


\bibitem{Berezin}
F.~A.~Berezin,  \textit{Introduction to superanalysis}.  
D.~Reidel Publishing Company, Holland, 1987.

\bibitem{BerLeites}
F.~A.~Berezin., Leites, \textit{Supermanifolds},
{Dokl. Akad. Nauk SSSR}, Vol. 224, no. {3}, 505-508, 1975. 

\bibitem{ccf} C.~Carmeli, L.~Caston, R.~Fioresi, 
{\it  Mathematical Foundation of Supersymmetry}, 
with an appendix with I. Dimitrov, EMS Ser. Lect. Math., European
Math. Soc., Zurich, 2011.

\bibitem{cf}
C. Carmeli, R. Fioresi,  
\textit{Superdistributions, analytic and algebraic super Harish-Chandra pairs}. 
Pacific J. Math. 263 (2013), no. 1, 29-51.

\bibitem{cfk1} C. Carmeli, R. Fioresi, S. D. Kwok, 
{\it SUSY structures, representations and
Peter-Weyl theorem for $S^{1|1}$}, 
 J. Geom. Phys. 95, 144-158, 2015.

\bibitem{cfk2} C. Carmeli, R. Fioresi, S. D. Kwok, 
{\it The Peter-Weyl theorem for $SU(1|1)$}. p-Adic Numbers Ultrametric Anal. 
Appl. 7, no. 4, 266-275, 2015.


\bibitem{dm} P.~Deligne, J.~Morgan,  {\it Notes on supersymmetry 
(following J.~Bernstein)},  in: ``Quantum fields and strings. 
A course for mathematicians'', Vol.~1, AMS, 1999.

\bibitem{dw1} R. Donagi, E. Witten, \textit{Supermoduli Space Is Not Projected},
 	arXiv:1304.7798.


\bibitem{fi} R. Fioresi, {\it Compact forms of Complex Lie Supergroups}
 J. Pure Appl. Algebra 218 (2014), no. 2, 228-236. 

\bibitem{fg1} R. Fioresi, F. Gavarini, \textit{Algebraic supergroups with 
Lie superalgebras of classical type}. J. Lie Theory 23 (2013), no. 1, 143-158.

\bibitem{fg2} R. Fioresi, F. Gavarini, \textit{Chevalley supergroups}. 
Mem. Amer. Math. Soc. 215 (2012), no. 1014.


\bibitem{fg3} R. Fioresi, F. Gavarini, \textit{On the construction of Chevalley 
supergroups}. Supersymmetry in mathematics and physics, 101-123, 
Lecture Notes in Math., 2027, Springer, Heidelberg, 2011.


\bibitem{fk1} R. Fioresi, S. D. Kwok, {\it On SUSY curves},
in "Advances in Lie Superalgebras"
M. Gorelik, Papi, P. (Eds.),
Springer INdAM Series, {\bf  7} (2014).

\bibitem{fk2} R. Fioresi, S. D. Kwok, {\it 
The Projective Linear Supergroup and the
SUSY-preserving automorphisms of $\bP^{1|1}$}, arXiv:1504.04492, 2015.

\bibitem{fl} R. Fioresi, M. A. Lledo, {\it The Minkowski
and conformal superspaces}, World Sci. Publishing, 2014.


\bibitem{rabin1} P. G. O. Freund, J. M. Rabin, 
\textit{Supertori are elliptic curves}, Comm.
Math. Phys. Vol. 114(1), 131-145 (1988).


\bibitem{kwok} S. D. Kwok, {\it Some results in supersymmetric algebraic
geometry}, Ph. D. thesis, UCLA, 2011.

\bibitem{Leites}
D.~A.~Leites,  \textit{Introduction to the theory of supermanifolds},  
Russian Math. Surveys \textbf{35}:~1 (1980), 1-64.

\bibitem{Levin}
A.~M.~Levin, \textit{Supersymmetric elliptic curves}, 
Funct. Analysis and Appl. 21 no. {3}, 243-244, 1987.

\bibitem{ma1}
Y.~I.~Manin,   \textit{Topics in non commutative geometry}; 
Princeton University Press, 1991. 

\bibitem{ma2}
Y.~I.~Manin,   \textit{Gauge field theory and complex geometry}; 
translated by N. Koblitz and J.R. King.  Springer-Verlag, 
Berlin-New York, 1988.



\bibitem{rabin2} J.~M.~Rabin, \textit{Super elliptic curves}, 
J. Geom. Phys. Vol.~15, 252-80 (1995).

\bibitem{vsv1}
V.~S.~Varadarajan,  \textit{Lie groups, Lie algebras, and their
  representations}.  Graduate Text in Mathematics.  
Springer-Verlag, New York, 1984.

\bibitem{vsv2} V.~S.~Varadarajan,  {\it Supersymmetry for
    mathematicians: an introduction},  Courant Lecture Notes  {\bf 1},
  AMS, 2004.

\bibitem{vi2} 
E.G. Vishnyakova, \textit{
On Complex Lie Supergroups and Homogeneous Split Supermanifolds}, 
Transformation Groups, Vol. 16, No. 1, 265-285, 2010.
 

%\bibitem{witten} E. Witten, \textit{Notes on super Riemann 
%surfaces and their moduli},  arXiv:1209.2459.


\end{thebibliography}
\end{document}